\documentclass[reqno, 11pt]{amsart}
\usepackage{amsmath,mathtools}
 \usepackage{amssymb}
\usepackage{amsthm}
\usepackage{amsmath}
\usepackage{times}
\usepackage{latexsym}
\usepackage[mathscr]{eucal}
\usepackage{float}

\numberwithin{equation}{section}
 
  \newtheorem{theorem}{Theorem}[section]
  \newtheorem{proposition}[theorem]{Proposition}
  \newtheorem{lemma}[theorem]{Lemma}
  \newtheorem{corollary}[theorem]{Corollary}

  \newtheorem{example}[theorem]{Example}

\title[Some properties of curvature tensors and foliations]{Some properties of curvature tensors and foliations of locally conformal almost K\"ahler manifolds}
\author[Ntokozo Sibonelo Khuzwayo, Fortun\'{e} Massamba]{Ntokozo Sibonelo Khuzwayo*, Fortun\'{e} Massamba**}
\newcommand{\acr}{\newline\indent}
\address{\llap{*\,} School of Mathematics, Statistics and Computer Science\acr
 University of KwaZulu-Natal\acr
 Private Bag X01, Scottsville 3209\acr 
South Africa}
\email{khuzwayon1@ukzn.ac.za} 
\thanks{}
\address{\llap{**\,} School of Mathematics, Statistics and Computer Science\acr
 University of KwaZulu-Natal\acr
 Private Bag X01, Scottsville 3209\acr
South Africa}
\email{massfort@yahoo.fr, Massamba@ukzn.ac.za} 
\thanks{}
\subjclass[2010]{Primary 53C15; Secondary 53C18, 53C55}

\keywords{Locally conformal manifold; Almost K\"ahler manifold; Foliation} 

\begin{document}

\begin{abstract}
We investigate a class of locally conformal almost K\"ahler structures and prove that, under some conditions, this class is a subclass of almost K\"ahler structures. We show that a locally conformal almost K\"ahler manifold admits a canonical foliation whose leaves are hypersurfaces with mean curvature vector field proportional to the Lee vector field. The geodesibility of the leaves is also characterized, and their minimality coincides with the incompressibility of the Lee vector field along the leaves.	
\end{abstract}
\maketitle

\section{Introduction}  

The study of manifolds whose metric is locally conformal to an almost K\"{a}hler metric is considered as one of the most interesting studies in the field of differential geometry (see \cite{avo} for details and references therein). This is because of its richness in the theory that is applicable in physics, algebraic geometry, symplectic geometry, etc. To our knowledge, locally conformal (almost) K\"{a}hler structures were first studied by P. Libermann \cite{LIBERMANN} in the 1950s. In 1966, A Gray \cite{zee} also contributed to the study by considering (almost) Hermitian manifolds whose metric is conformal to a local (almost) K\"{a}hler metric. However, globally conformal (almost) K\"ahler manifolds share the same topological properties with locally conformal (almost) K\"{a}hler manifolds \cite{izu}. It is therefore provocative to consider those almost Hermitian manifolds whose metric is locally conformal to an almost K\"ahler metric. The difference between locally conformal K\"{a}hler manifolds and locally conformal almost K\"ahler manifolds is the condition of integrability of an almost complex structure. This is equivalent to an almost complex structure being parallel with respect to a globally defined connection or the vanishing of a Nijenhuis tensor. Therefore, the geometric properties which do not depend on the almost complex structure will apply to both of these manifolds.

Libermann defined a locally conformal (almost) K\"{a}hler metric as a metric $g$ at which in the neighborhood of each point of an almost Hermitian manifold, it is locally conformal to an (almost) K\"{a}hler metric. 

In this paper, we investigate some properties of curvature tensors and foliations of locally conformal almost K\"ahler manifolds. For the foliations, we pay attention to the ones that arise naturally when the Lee form is nowhere vanishing. The paper is organized as follows. In Section \ref{DefinLocal}, we recall the definition of locally conformal almost K\"ahler structures supported by an example. In Section \ref{Curva}, we deal with curvature tensors. The latter generalizes those found by Olszak in \cite{Ols}. Under some conditions, we prove that a class of locally conformal almost K\"ahler structures is a subclass of almost K\"ahler structures. The Section \ref{Foliations} is devoted to the canonical foliations that arise for non-vanishing Lee form. We prove that these foliations are Riemannian if and only if the Lee vector field is auto-parallel. We also prove that the locally conformal almost K\"ahler manifolds contain leaves with mean curvature vector field proportional to the Lee vector field, and their geodesibility coincides with the Killing condition of the Lee vector field. The latter is incompressible along the leaves if and only if the leaves are minimal.

\section{Locally conformal almost K\"{a}hler metrics}\label{DefinLocal}

Let $M$ be a $2n$-dimensional almost Hermitian manifold with the metric $g$ and the almost complex structure $J$ satisfying
$$
J^{2} = - \mathbb{I},\;\;\; g(J X, J Y)= g(X, Y),
$$
for any vector fields $X$ and $Y$ tangent to $M$, where $\mathbb{I}$ stands for the identity transformation of tangent bundle $T M$. Then for any vector fields $X$ and $Y$, the tensor
\begin{equation}
\Omega(X,Y)=g(X,JY),
\end{equation}
defines the fundamental $2$-form of $M$ which is non-degenerate and gives an almost symplectic structure on $M$. If $\Omega$ is closed, i.e., $d\Omega=0$, then $(M, J, g)$ is called an \textit{almost K\"ahler manifold} \cite{CaDeDeO}.

Now, let $(M,J, g)$ be a $2n$-dimensional almost Hermitian manifold. Such manifold is said to be locally conformal almost K\"ahler manifold \cite{izu} if there is an open covering $\{ U_t \}_{t \in I}$ of $M$ and a family $\{f_{t}\}_{t\in I}$ of $\mathcal{C}^\infty$-functions $f_t: U_t \rightarrow \mathbb{R}$ such that, for any $t \in I$, the metric form 
	\begin{equation}\label{MetricTrans}
	g_t=\exp(-f_t)g_{|_{U_{t}}},
	\end{equation}
	is an \textit{almost K\"ahler metric}.	

If the structures $(J, g_{t})$ defined in (\ref{MetricTrans}) are K\"ahler, then $(M,J, g)$ is called locally conformal K\"ahler. Moreover, a locally conformal almost K\"ahler manifold $M$ is almost K\"ahler if and only if $d f_{t}=0$.

The Lee form is important because it characterizes locally conformal almost K\"ahler manifolds. Locally confornal almost K\"ahler manifolds were characterized by Vaisman in \cite{izu}. This is stated follows: \textit{An almost Hermitian manifold $(M, J, g)$ is a locally conformal almost K\"{a}hler manifold if and only if there exists a  $1$-form $\omega$ such that
	\begin{equation} \label{famous}
	d\Omega = \omega \wedge \Omega \; \;\mbox{and}\; \; d\omega=0.
	\end{equation} 
}
\begin{example}\label{example1}
	{\rm We consider the 4-dimensional manifold $M^{4}=\{p\in \mathbb{R}^{4}| x_1\neq 0, \; x_{2}>0 \},$ where $p=(x_1,x_2,y_1,y_2)$ are the standard coordinates in $\mathbb{R}^{4}$. The vector fields,
		$$
		X_{i}=x_{2}\frac{\partial}{\partial x_{i}}, \;\;  Y_{i}=\frac{1}{x_{2}^{3}}\frac{\partial}{\partial y_{i}}, \;\;\mbox{for}\;\;i=1,2, 
		$$
		are linearly independent at each point of $M$. Let $g$ be the Riemannian metric on $M$ defined by $g(X_{i}, X_{j}) = g(Y_{i}, Y_{j}) = \delta_{ij}$, where $\delta_{ij}$ is the Kronecker symbol, $g(X_{i}, Y_{j}) =0$. That is, the form of the
		metric becomes
		$$
		g=\frac{1}{x_{2}^2}(dx^2_{1}+dx^2_{2})+x_{2}^6(dy^2_{1}+dy^2_{2}).
		$$
		Let $J$ be the $(1, 1)$-tensor field defined by,
		$J X_{1}  = Y_{1}$, $J X_{2}  = -Y_{2}$, $J Y_{2}= X_{2}$, $J Y_{1} = - X_{1}$. Thus, $(J, g)$ defines an almost Hermitian structure on $M^{4}$.  The non-zero component of the fundamental $2$-form $J$ is  
		$$
		\Omega(\frac{\partial}{\partial x_{1}},\frac{\partial}{\partial y_{1}}) =  -\frac{1}{x_{2}^{2}} \;\;\mbox{and}\;\; \Omega(\frac{\partial}{\partial x_{2}},\frac{\partial}{\partial y_{2}}) =  \frac{1}{x_{2}^{2}} 
		$$
		and we have 
		$$
		\Omega =  \frac{1}{x_{2}^{2}}\left\{- dx_{1}\wedge dy_{1} + dx_{2}\wedge dy_{2}\right\}.
		$$
		Its differential gives
		$$
		d\Omega = \frac{2}{x_{2}^{3}} dx_{1}\wedge dy_{1}\wedge dx_{2}. 
		$$
		By letting $$\omega =  - \frac{1}{x_{2}} dx_{2},$$ we have, $$d\Omega = 2\omega\wedge \Omega.$$ It is easy to see that $d\omega=0$ and the dual vector field $B$ is given by 
		$$
		B=- X_{2}.
		$$ Let us consider the open neighborhood $U$ of $M$ given by  $U=\{p\in M^{4}| x_{2}> 0 \}$, and there exists a differentiable function $f$ on $U$ such that $\omega= d\,f$, where $f= - \ln (x_{2})$. By the characterization given in (\ref{famous}) above-mentioned, $(M^{4}, J, g)$ is a locally conformal almost K\"ahler manifold. } 
\end{example} 
Next, wish to study the relationship of the Levi-Civita connections induced by the locally conformal K\"ahler metric $g_t$ and $g$.

Throughout this paper,  $\Gamma(\Xi)$ will denote the $\mathcal{F}(M)$-module of differentiable sections of a vector bundle $\Xi$. 

Let $\nabla$ and $\nabla^{t}$ be the Levi-Civita connections associated with the metrics $g$ and $g_{t}$, respectively. As is well-known, they are connected by   
	\begin{equation} \label{vic}
	\nabla^t_X Y = \nabla_X Y -\frac{1}{2} \left\{ \omega(X)Y+\omega(Y)X -g(X,Y)B \right\},
	\end{equation}
	for all $X,Y \in\Gamma(TM)$.

\section{Curvature relations of locally conformal almost K\"{a}hler metrics}\label{Curva}

Let $(M, J,g)$ be a $2n$-dimensional almost Hermitian manifold. Here we keep the formalism of local transformations and others formulas defined in the previous Section. For the Riemann curvature $R$ of a metric $g$, we use the following convention 
\begin{equation}
R(X,Y,Z,W)=g(R(X,Y)Z,W),
\end{equation}
where 
\begin{equation} \label{CURVACOVEN}
R(X,Y)Z=\nabla_X\nabla_YZ- \nabla_Y\nabla_XZ-\nabla_{[X,Y]}Z,
\end{equation}
for any vector field $X$, $Y$ and $Z$ on $M$.

Let $\{E_i\}_{1 \leq i \leq 2n}$ be the orthonormal basis with respect to $g$. The Ricci curvature tensor $\rho$ and the scalar curvature $\tau$ are given by
\begin{align}
\rho(X,Y)=\sum_{i=1}^{2n}R(E_i, X, Y,E_i) \;\;\mbox{and}\;\; \tau = \sum_{i=1}^{2n}\rho(E_{i}, E_{i}).
\end{align}
Now we consider the Ricci $*$-curvature tensor $\rho^{*}$ and the scalar $*$-curvature $\tau^{*}$ defined by   
\begin{align}
\rho^{*}(X,Y)=\sum_{i=1}^{2n}R(E_i, X, JY, JE_i) \;\;\mbox{and}\;\; \tau^{*} = \sum_{i=1}^{2n}\rho^{*}(E_{i}, E_{i}).
\end{align}
Similarly, the curvatures corresponding to the metric $g_{t}$ will be denoted by $R^{t}$, $\rho^{t}$, $\tau^{t}$, $\rho^{t*}$ and $\tau^{t*}$, respectively.
\begin{lemma} \label{bandura}
	Let $(M,J,g)$ be a locally conformal almost K\"{a}hler manifold. Then the curvature tensors $R^t$ and $R$ with respect to the metrics $g_t$ and $g$, respectively, are related as
	\begin{align} \label{zar}
	&R^t(X,Y)Z = R(X,Y)Z +\frac{1}{2} \left\{(\nabla_Y \omega)Z +\frac{1}{2} \omega(Y) \omega(Z)\right\}X\nonumber\\
	&-\frac{1}{2} \left\{ (\nabla_X \omega)Z + \frac{1}{2} \omega(X) \omega(Z) \right\}Y +\frac{1}{2}g(Y,Z)\left\{ \nabla_X B + \frac{1}{2} \omega(X)B \right\}\nonumber\\
	& -\frac{1}{2}g(X,Z)\left\{ \nabla_Y B +\frac{1}{2} \omega(Y)B \right\}-\frac{||B||^2}{4} \left\{g(Y,Z)X-g(X,Z)Y \right\},
	\end{align}
	where $||B||^{2} = g(B, B)$.
	\begin{proof}
		Using the convention in (\ref{CURVACOVEN}) for the curvature tensors $R^{t}$ and $R$ and the relation (\ref{vic}), and for any vector fields $X$, $Y$ and $Z$ tangent to $M$, the expressions
		\begin{align}\label{NAblaNabla1}
		\nabla^t_X \nabla^t_YZ &=\nabla_X \nabla_YZ - \frac{1}{2}\omega(X)\nabla_{Y}Z - \frac{1}{2}\omega(\nabla_{Y}Z)  +  \frac{1}{2} g(X, \nabla_{Y}Z)B\nonumber\\
		&  - \frac{1}{2}X(\omega(Y))Z - \frac{1}{2}\omega(Y)\nabla_{X}Z + \frac{1}{4}\omega(X)\omega(Y)Z+ \frac{1}{4}\omega(Y)\omega(Z)X\nonumber\\
		&-\frac{1}{4}\omega(Y)g(X, Z)B - \frac{1}{2}X(\omega(Z))Y- \frac{1}{2}\omega(Z)\nabla_{X}Y+ \frac{1}{4}\omega(X)\omega(Z)Y\nonumber\\
		&+ \frac{1}{4}\omega(Y)\omega(Z)X-\frac{1}{4}\omega(Z)g(X, Y)B + \frac{1}{2}X(g(Y, Z))B + \frac{1}{2}g(Y, Z)\nabla_{X}B\nonumber\\
		& - \frac{1}{4}||B||^{2}g(Y, Z)X,
		\end{align}
		It is worth noting that
		\begin{equation}\label{NAblaNabla3}
		\nabla^t_{[X,Y]}Z=\nabla_{[X,Y]}Z-\frac{1}{2}\omega([X,Y])Z -\frac{1}{2}\omega(Z)[X, Y] + \frac{1}{2}g([X,Y],Z)B. 
		\end{equation}
		Putting the pieces (\ref{NAblaNabla1}) and (\ref{NAblaNabla3}) together, one obtains the relation (\ref{zar}). This completes the proof.
	\end{proof}
\end{lemma}
Next, from the above Lemma, we define $(0,2)$-tensor field $P$ by 
\begin{equation}\label{SYMMPAY}
P(X,Y)=(\nabla_X \omega)Y+\frac{1}{2} \omega(X)\omega(Y)-\frac{1}{4}||B||^2g(X,Y),
\end{equation}
and this trace is given by
\begin{equation}
\mathrm{trace} P=  \mathrm{div}B- \frac{1}{2}(1-n)||B||^{2}.
\end{equation}
\begin{lemma}
	The $(0,2)$-tensor field $P$ is symmetric.
\end{lemma}
\begin{proof}
	For any vector fields $X$ and $Y$ tangent to $M$ and since $\omega$ is closed, we have
	\begin{align}
	P(X, Y) & = (\nabla_Y \omega)X+\frac{1}{2} \omega(X)\omega(Y)-\frac{1}{4}||B||^2g(X,Y)\nonumber\\
	& = Y(\omega(X)) - \omega(\nabla_{Y}X)+\frac{1}{2} \omega(X)\omega(Y)-\frac{1}{4}||B||^2g(X,Y)\nonumber\\
	& =Y(\omega(X)) - \omega([Y, X])  - \omega(\nabla_{X}Y) +\frac{1}{2} \omega(X)\omega(Y)-\frac{1}{4}||B||^2g(X,Y)\nonumber\\
	& =(\nabla_X \omega)Y+\frac{1}{2} \omega(X)\omega(Y)-\frac{1}{4}||B||^2g(X,Y),\nonumber
	\end{align} 
	which completes the proof.
\end{proof}
The Lie derivative $g$ with respect to the vector field $B$ gives, for any vector fields $X$ and $Y$,
\begin{align}\label{LIED}
(L_{B}g)(X, Y) & =  X(g(B, Y)) - g([B, X], Y) - g(X, [B, Y])\nonumber\\
& = (\nabla_{X}\omega)Y + (\nabla_{Y}\omega)X= 2(\nabla_{X}\omega)Y.
\end{align} 
The last equality of (\ref{LIED}) follows from the fact that the smooth 1-form $\omega$ is closed.
\begin{lemma}
	The dual vector field $B$ of $\omega$ preserves the metric $g$ if and only if the Lee form $\omega$ is $\nabla$-parallel.
\end{lemma}

The Riemannian curvatures are related by, for any $X$, $Y$, $Z$ and $W$ tangent to $M$,
\begin{align}
\exp(f_{t})R^{t}(X, Y, Z, W) & = R(X, Y, Z, W) + \frac{1}{2} \left\{g(X, W)P(Y, Z) - g(Y, W)P(X, Z)\right\}\nonumber\\
& + \frac{1}{2}  \left\{g(Y, Z)P(X, W) - g(X, Z)P(Y, W)\right\}.
\end{align}
Let $\{ E_i\}$ be the orthonormal basis with respect to $g$. Then, we have 
\begin{align}\nonumber
g(E_i,E_j)= \left\{ \begin{array}{cc} 
1, & \hspace{5mm} \mathrm{if }\;\; i=j, \\
0, & \hspace{5mm} \mathrm{if }\;\;  i\neq j. \\
\end{array} \right.
\end{align}
Let $E^t_i=\exp(f_t)^\frac{1}{2}E_i$, for any $i=1,2,\cdots, 2n$. Therefore, we have the following.
\begin{lemma}\label{Frame}
	The frame $\{ E^{t}_i\}_{1\le i\le 2n}$ is the orthonormal basis with respect to $g_{t}$.	
\end{lemma}
The following identities generalize the ones given in \cite[p.216]{Ols}.
\begin{lemma}
	The Ricci curvature  tensors $\rho^t$ and $\rho$ with respect to $g_t$ and $g$, respectively, are related by
	\begin{align}\label{RhoForm}
	\rho^t(X,Y)=\rho(X,Y)+ (n-1)P(X,Y)+\frac{1}{2} g(X,Y)\,\mathrm{trace}\, P.
	\end{align}	
\end{lemma}
\begin{proof}
	Using the Lemma \ref{Frame} and for any vector fields $X$ and $Y$ tangent to $M$, one has
	\begin{align}
	\rho^t(X,Y)&=\sum_{i=1}^{2n}R^{t}(E^{t}_i, X, Y, E^{t}_i)=\sum_{i=1}^{2n}\exp(f_{t}) R^{t}(E_i, X, Y, E_i)\nonumber\\ 
	&=\sum_{i=1}^{2n}  R (E_i, X, Y, E_i) + \frac{1}{2}\left\{\sum_{i=1}^{2n} g(E_{i}, E_{i}) P(X, Y) -  \sum_{i=1}^{2n} g(X, E_{i}) P(E_{i}, Y) \right\}\nonumber\\ 
	&+ \frac{1}{2}\left\{\sum_{i=1}^{2n} g(X, Y)P(E_{i}, E_{i}) -  \sum_{i=1}^{2n} g(E_{i}, Y) P(X, E_{i}) \right\}\nonumber\\ 
	&= \rho(X, Y) + (n-1)  P(X, Y)+ \frac{1}{2} g(X, Y)\mathrm{trace}\,P,\nonumber
	\end{align}
	which completes the proof.
\end{proof}
Also,  corresponding Ricci $*$-curvatures are related by 
\begin{equation} \label{ntrace}
\rho^t{^*}(X,Y)=\rho^*(X,Y)+\frac{1}{2} \left\{  P(X,Y)+P(JX,JY) \right\}.
\end{equation}
\begin{corollary}
	The scalar curvatures $\tau^t$ and $\tau$ are related by
	\begin{equation} \label{bhoza}
	\exp(-f_t)\tau_t = \tau+ (2n-1)\left\{\mathrm{div}B- \frac{1}{2}(1-n)||B||^{2}\right\}.
	\end{equation}
	\begin{proof}
		Using the Lemma \ref{Frame}, the scalar curvature $\tau^t$, we have 
		\begin{equation} \label{nangu}
		\tau^t=\sum_{i=1}^{2n}  \rho^t(E_i^t,E_i^t)=\exp(f_t)  \sum_{i=1}^{2n} \rho^t(E_i,E_i).
		\end{equation} 
		Then, applying Equation (\ref{RhoForm}) into (\ref{nangu}), we get
		\begin{align} 
		\exp(-f_t) \tau^t&= \sum_{i=1}^{2n}\rho(E_i,E_i)+(n-1)\sum_{i=1}^{2n}P(E_i,E_i)+ n\,\mathrm{trace}P   \nonumber \\ 
		&= \tau+(2n-1)\, \mathrm{trace}P   \nonumber \\ 
		&=\tau+ (2n-1)\left\{\mathrm{div}B- \frac{1}{2}(1-n)||B||^{2}\right\}.\nonumber 
		\end{align}
		Therefore, 
		\begin{equation}
		\exp(-f_t) \tau^t= \tau+ (2n-1)\left\{\mathrm{div}B- \frac{1}{2}(1-n)||B||^{2}\right\},\nonumber
		\end{equation}
		which completes the proof.
	\end{proof}
\end{corollary}
Now if we consider a relation between the scalar $*$-curvature $\tau^t{^*}$ and $\tau^{*}$, we get:
\begin{corollary}
	The scalar $*$-curvatures $\tau^t{^*}$ and $\tau^*$ are related by
	\begin{equation} \label{bhozakazi}
	\exp(-f_t) \tau^t{^*}=\tau^* +\mathrm{div}B + (n-1)||B||^2.
	\end{equation}
	\begin{proof}
		The scalar $*$-curvature $\tau^t{^*}$ is given by 
		\begin{equation} \label{thixo}
		\tau^t{^*}=\sum_{i=1}^{2n}\rho^t{^*}(E^t_i,E^t_i)=\exp(f_t)\sum_{i=1}^{2n}\rho^t{^*}(E_i,E_i).
		\end{equation}
		Now applying the relation (\ref{ntrace}) into (\ref{thixo}), we compute
		\begin{align} \nonumber
		\exp(-f_t) \tau^t{^*}&=\sum_{i=1}^{2n}\rho^t{^*}(E_i,E_i)\nonumber \\ 
		&=\sum_{i=1}^{2n} \rho^*(X,Y)+ \frac{1}{2}\sum_{i=1}^{2n} \left\{P(E_i,E_i)+ P(JE_i,jE_i) \right\} \nonumber\\ 
		&= \tau^*+  \mathrm{div}B +(n-1)||B||^2. \nonumber
		\end{align}
		Hence,
		\begin{equation}
		\exp(-f_t) \tau^t{^*}= \tau^*+\mathrm{div}B + (n-1)||B||^2, \nonumber
		\end{equation}
		as required.
	\end{proof}
\end{corollary}
Gray in \cite{gray} considered some curvature identities for Hermitian and almost Hermitian manifolds. Let $\mathcal{L}$ be the class of almost Hermitian manifolds as defined in \cite{gray}. Then the manifold under consideration is an element of the class $\mathcal{L}$. Now consider as in \cite{gray} the curvature operator $R^{t}$ of a locally conformal almost K\"{a}hler manifold $M$:
\begin{align}
(1) \quad R^{t}(X, Y, Z, W) & = R^{t}(X, Y, J Z, J W)\nonumber\\
(2) \quad R^{t}(X, Y, Z, W) & -R^{t}(JX, JY, Z, W) = R^{t}(J X, Y, J Z, W)\nonumber\\
&+ R^{t}(J X, Y,  Z, J W),\nonumber\\
(3) \quad R^{t}(X, Y, Z, W)& = R^{t}(JX, JY, J Z, J W),\nonumber
\end{align}
for any $X$, $Y$, $Z$ and $W$ tangent to $M$. The item (1) is called K\"ahler identity if $M$ is locally conformal K\"{a}hler manifold (see \cite{gray} for more details and reference therein).

We will focus, throughout the rest of this note, on the item (1). Using further notations as in \cite{gray}, we denoted by $\mathcal{L}_{i}$ the subclass of manifolds whose curvature operator $R^{t}$ satisfies identity (i). Here (i) may be either the item (1), (2) or (3) above. As in \cite{gray}, it is easy to see that 
$$
\mathcal{L}_{1}\subseteq\mathcal{L}_{2} \subseteq\mathcal{L}_{3} \subseteq\mathcal{L}.
$$
Therefore we have the following result.
\begin{lemma}\label{YABIEN}
	If a locally conformal almost K\"ahler manifold is in a class $\mathcal{L}_{1}$, then the equality holds:
	\begin{align}\label{DIFFRe}
	\tau^{*}-\tau = 2(n-1)\,\mathrm{trace}\, P.
	\end{align}
\end{lemma}
\begin{proof}
	The proof follows from a straightforward calculation using the fact that, for any vector fields $X$ and $Y$ tangent to $M$, we have
	\begin{align}\label{RICCCITrans}
	\rho^{t}(X, Y) & = \sum_{i=1}^{2n}\exp(f_{t})R^{t}(E_{i}, X, Y, E_{i})\nonumber\\
	&= \sum_{i=1}^{2n}\exp(f_{t})R^{t}(E_{i}, X, JY, JE_{i})= \rho^{*t}(X, Y),
	\end{align}
	which leads to 
	$$
	(\rho^{*}-\rho)(X, Y) = (n-\frac{3}{2})P(X, Y) + \frac{1}{2}g(X, Y)\mathrm{trace}\,P - \frac{1}{2}P(JX, JY).
	$$
	This completes the proof.
\end{proof}
The relation (\ref{RICCCITrans}) leads to 
\begin{equation} \nonumber
\tau^t  =\sum_{i=1}^{2n}\sum_{j=1}^{2n}R^t(E_j,E_i,E_i,E_j) =\sum_{i=1}^{2n}\sum_{j=1}^{2n}R^t(E_j,E_i,JE_i,JE_j)  =\tau^t{^*}.
\end{equation}
\begin{theorem}
	Let $(M,J,g)$ be a $2n$-dimensional compact locally conformal almost K\"{a}hler manifold with $n>1$ and contained in $\mathcal{L}_{1}$. If
	$$
	\tau^{*}=\tau,
	$$
	then $(M,J,g)$ is an almost K\"{a}hler manifold.
\end{theorem}
\begin{proof}
	By Lemma \ref{YABIEN}, we have $\tau^{*}-\tau = 2(n-1)\mathrm{trace}\, P$, with $\mathrm{trace} P=  \mathrm{div}B- \frac{1}{2}(1-n)||B||^{2}$. Taking into account this, integrating the relation (\ref{DIFFRe}) and using Green's Theorem, we have 
	$ 
	0=\int_{M}\{\tau^{*}-\tau \}= (n-1)^{2}\int_{M}||B||^{2}.
	$  
	Hence, under our assumption, we obtain $B=0$. Therefore $\omega=0$ identically on $M$. Hence $(M.J,g)$ is an almost  K\"{a}hler manifold.
\end{proof}

As an example for this Theorem, we have compact flat locally almost K\"ahler manifolds. Compact flat manifolds have been detailed in \cite{Charlap} and reference therein.

\section{Lee form and canonical foliations}\label{Foliations}

Let $(M, J, g)$ be a locally conformal almost K\"ahler manifold and assume that the Lee form $\omega$ is never vanishing on $M$. Then $\omega=0$ defines on $M$ an integrable distribution, and hence a foliation $\mathcal{F}$, on $M$ (see \cite{KIM} for more details and reference therein). 

Let $D:=\ker\omega$ be the distribution on $M$ and $D^{\perp}$ be the distribution spanned the vector field $B$.Then, we have the following decomposition
\begin{equation}\label{decomp}
T M = D\oplus D^{\perp},
\end{equation} 
where $\oplus$ denotes the orthogonal direct sum. By the decomposition (\ref{decomp}), any $X\in\Gamma(T M)$ is written as
\begin{equation}
X = Q X + Q^{\perp}X,
\end{equation}
where $Q$ and $Q^{\perp}$ are the projection morphisms of $T M$ into $D$ and $D^{\perp}$, respectively. Here, it is easy to see that $Q^{\perp}X= \frac{1}{||B||^{2}}\omega(X)B$ and 
$$ 
X = Q X + \frac{1}{||B||^{2}}\omega(X)B.
$$
Let $\mathcal{F}$ be a foliation on a locally conformal almost K\"ahler manifold $(M, J, g)$ of codimension 1. The metric $g$ is said to be \textit{bundle-like} for the foliation $\mathcal{F}$ if the induced metric on the transversal distribution $D^{\perp}$ is parallel with respect to the intrinsic connection on $D^{\perp}$. This is true if and only if the Levi-Civita connection $\nabla$ of $(M, J,  g)$ satisfies (see \cite{bf} and \cite{To} for more details):  
\begin{equation}\label{BundleLike}
g(\nabla_{Q^{\perp}Y}QX, Q^{\perp}Z) +  g(\nabla_{Q^{\perp}Z}QX, Q^{\perp}Y)=0,  
\end{equation}
for any $X$, $Y$, $Z\in\Gamma(T M)$.  A foliation $\mathcal{F}$ is said to be \textit{Riemannian} on $(M, J,  g)$ if the Riemannian metric $g$ on $M$ is bundle-like for $\mathcal{F}$.  

Let $\mathcal{F}^{\perp}$ be the orthogonal complementary foliation generated by $B$. Now we provide necessary and sufficient conditions for the metric on an locally conformal  almost K\"ahler manifold to be bundle-like for foliations $\mathcal{F}$ and $\mathcal{F}^{\perp}$.
Therefore
\begin{theorem}\label{Mass}
	Let $(M, J, g)$ be a locally conformal almost K\"ahler manifold and let $\mathcal{F}$ be a foliation on $M$ of codimension 1. Then the following assertions are equivalent:
	\begin{enumerate}
		\item[(i)] The foliation $\mathcal{F}$ is Riemannian.
		\item[(ii)] The Lee vector field $B$ is auto-parallel with respect to $\nabla$, that is, $$\nabla_{B}B=  B\left(\ln(||B||)\right)B.$$
	\end{enumerate} 
\end{theorem}
\begin{proof}
For any $X$, $Y$, $Z\in\Gamma(T M)$, we have $Q^{\perp}Y=\frac{1}{||B||^{2}}\omega(Y)B$, $Q^{\perp}Z=\frac{1}{||B||^{2}}\omega(Z)B $ and the left-hand side of (\ref{BundleLike}) gives 
	\begin{equation}
	g(\nabla_{Q^{\perp}Y}QX, Q^{\perp}Z)+ g(\nabla_{Q^{\perp}Z}QX, Q^{\perp}Y) =\frac{2}{||B||^{2}}\omega(Y)\omega(Z) \omega(\nabla_{B}QX),\nonumber
	\end{equation}
	for which the equivalence follows.  
\end{proof}
Let $M'$ be a leaf of the distribution $D$. Since $M'$ is a submanifold of $M$ and for any $X$, $Y\in\Gamma(T M')$, we have
\begin{align}\label{gauss}
\nabla_{X} Y & = \nabla_{X}'Y + \alpha(X, Y),\\\label{weint}
\nabla_{X}B & = - A_{B}X + {\nabla'}^{\perp}_{X}B,
\end{align}
where $\nabla'$ and $\alpha$ are the Levi-Civita connection and the second fundamental form of $M'$, respectively.  Here $A_{B}$ is the shape operator with respect to $B$. On the other hand,  we have 
$g(\nabla_{X}B, B)= X(\omega(B)) - g(\nabla_{X}B, B)$, hence 
$$
g({\nabla'}^{\perp}_{X}B, B)=\frac{1}{2}X(\omega(B)),
$$
for any $X\in\Gamma(T M')$. Therefore, the Weingartem formula (\ref{weint}) becomes 
\begin{equation}\label{weint1}
\nabla_{X}B  = - A_{B}X + \frac{1}{2}X(\omega(B))B.
\end{equation} 
\begin{proposition}\label{LeavesTotGeo}
	Let $(M, J, g)$ be a locally conformal almost K\"ahler manifold. Then, the mean curvature vector field $H'$ of the leaves of the integrable distribution $D$ defined in (\ref{decomp}) is given by
	$$
	H'= \frac{1}{2n-1}\left(\mathrm{div}_{|_{M'}}B\right)B.
	$$
	Moreover, these leaves are totally geodesic hypersurfaces of $M$ if and only if the dual vector field $B$ of $\omega$ preserves their metrics.
\end{proposition}
\begin{proof}
	Let $M'$ be a leaf of the integrable distribution $D$. Using (\ref{gauss}) and (\ref{weint1}), the second fundamental form of $M'$ gives
	\begin{equation}
	\alpha(X, Y)   = g(A_{B}X, Y)B  =  g(\nabla_{X}B, Y)B,\nonumber
	\end{equation} 
for any $X$, $Y\in\Gamma(T M')$. Fixing a local orthonormal frame $\{e_{1},\cdots,e_{2n-1} \}$ in $T M'$, one has,
	\begin{equation}
	H  = \frac{1}{2n-1}  \sum_{i=1}^{2n-1}\alpha(e_{i}, e_{i}) =\frac{1}{2n-1}\left(\mathrm{div}_{|_{M'}}B\right)B.\nonumber
	\end{equation}
	The last assertion follows and this completes the proof.
\end{proof}
Therefore we have the following result.
\begin{corollary}
	Let $(M, J, g)$ be a locally conformal almost K\"ahler manifold. Then, the leaves $M'$ of the distribution $D$ in (\ref{decomp}) are minimal if and only if the dual vector field $B$ is incompressible along $M'$.
\end{corollary}	

\section*{Data Availability}
No data were used to support this study.

\section*{Conflicts of Interest}
The authors declare that they have no conflicts of interest regarding the publication of this paper.

\section*{Acknowledgments}
This work is based on the research supported wholly / in part by the National Research Foundation of South Africa (Grant Numbers: 95931 and 106072).

\end{document}